\documentclass[fleqn,12pt,twoside]{article}

\usepackage[headings]{espcrc1}
\readRCS
$Id: espcrc1.tex,v 1.2 2004/02/24 11:22:11 spepping Exp $
\usepackage{graphicx}
\usepackage[figuresright]{rotating}

\newtheorem{theorem}{Theorem}

\newtheorem{corollary}[theorem]{Corollary}

\newtheorem{lemma}[theorem]{Lemma}

\newenvironment{proof}[1][Proof.]{\begin{trivlist}
\item[\hskip \labelsep {\bfseries #1}]}{\end{trivlist}}

\newenvironment{acknowledgement}[1][Acknowledgement]{\begin{trivlist}
\item[\hskip \labelsep {\bfseries #1}]}{\end{trivlist}}

\newcommand{\AmS}{{\protect\the\textfont2
  A\kern-.1667em\lower.5ex\hbox{M}\kern-.125emS}}

\usepackage{amsmath}
\usepackage{amsfonts}
\usepackage{amssymb}
\title{Interval edge-colorings of composition of
graphs}

\author{P.A. Petrosyan\address[MCSD]{Department of Informatics and Applied Mathematics,\\
Yerevan State University, 0025, Armenia}%
\address{Institute for Informatics and Automation Problems,\\
National Academy of Sciences, 0014, Armenia}%
\thanks {email: pet\_petros@ipia.sci.am},
        H.H. Tepanyan
\address{Stanford University, Stanford, CA 94305, United States}%
\thanks{email: tehayk@stanford.edu}}

\runtitle{Interval edge-colorings of composition of
graphs}\runauthor{P.A. Petrosyan, H.H. Tepanyan}

\begin{document}

\maketitle

\begin{abstract}
An edge-coloring of a graph $G$ with consecutive integers
$c_{1},\ldots,c_{t}$ is called an \emph{interval $t$-coloring} if
all colors are used, and the colors of edges incident to any vertex
of $G$ are distinct and form an interval of integers. A graph $G$ is
interval colorable if it has an interval $t$-coloring for some
positive integer $t$. The set of all interval colorable graphs is
denoted by $\mathfrak{N}$. In 2004, Giaro and Kubale showed that if
$G,H\in \mathfrak{N}$, then the Cartesian product of these graphs
belongs to $\mathfrak{N}$. In the same year they formulated a
similar problem for the composition of graphs as an open problem.
Later, in 2009, the first author showed that if $G,H\in
\mathfrak{N}$ and $H$ is a regular graph, then $G[H]\in
\mathfrak{N}$. In this paper, we prove that if $G\in \mathfrak{N}$
and $H$ has an interval coloring of a special type, then $G[H]\in
\mathfrak{N}$. Moreover, we show that all regular graphs, complete
bipartite graphs and trees have such a special interval coloring. In
particular, this implies that if $G\in \mathfrak{N}$ and $T$ is a
tree, then $G[T]\in \mathfrak{N}$.\\

Keywords: edge-coloring, interval coloring, composition of graphs,
complete bipartite graph, tree.
\end{abstract}

\section{Introduction}\

All graphs considered in this paper are finite, undirected, and have
no loops or multiple edges. Let $V(G)$ and $E(G)$ denote the sets of
vertices and edges of $G$, respectively. For a graph $G$, by
$\overline G$ we denote the complement of the graph $G$. The degree
of a vertex $v\in V(G)$ is denoted by $d_{G}(v)$, the maximum degree
of $G$ by $\Delta(G)$, and the chromatic index of $G$ by
$\chi^{\prime}(G)$. The terms and concepts that we do not define can
be found in \cite{AsrDenHag,HamImrichKlavzar,Kubale,West}.

A proper edge-coloring of a graph $G$ is a coloring of the edges of
$G$ such that no two adjacent edges receive the same color. A proper
edge-coloring of a graph $G$ with consecutive integers
$c_{1},\ldots,c_{t}$ is an \emph{interval $t$-coloring} if all
colors are used, and the colors of edges incident to each vertex of
$G$ are form an interval of integers. A graph $G$ is \emph{interval
colorable} if it has an interval $t$-coloring for some positive
integer $t$. The set of all interval colorable graphs is denoted by
$\mathfrak{N}$. The concept of interval edge-coloring of graphs was
introduced by Asratian and Kamalian \cite{AsrKam} in 1987. In
\cite{AsrKam}, they proved that if $G\in \mathfrak{N}$, then
$\chi^{\prime}\left(G\right)=\Delta(G)$. Asratian and Kamalian also
proved \cite{AsrKam,AsrKamJCTB} that if a triangle-free graph $G$
admits an interval $t$-coloring, then $t\leq \left\vert
V(G)\right\vert -1$. In \cite{Kampreprint,KamDiss}, Kamalian
investigated interval colorings of complete bipartite graphs and
trees. In particular, he proved that the complete bipartite graph
$K_{m,n}$ has an interval $t$-coloring if and only if
$m+n-\gcd(m,n)\leq t\leq m+n-1$, where $\gcd(m,n)$ is the greatest
common divisor of $m$ and $n$. In \cite{PetDM}, Petrosyan
investigated interval colorings of complete graphs and hypercubes.
In particular, he proved that if $n\leq t\leq
\frac{n\left(n+1\right)}{2}$, then the hypercube $Q_{n}$ has an
interval $t$-coloring. Later, in \cite{PetKhachTan}, it was shown
that the hypercube $Q_{n}$ has an interval $t$-coloring if and only
if $n\leq t\leq \frac{n\left(n+1\right)}{2}$. In \cite{Seva},
Sevast'janov proved that it is an $NP$-complete problem to decide
whether a bipartite graph has an interval coloring or not. In papers
\cite{AsrKam,AsrKamJCTB,GiaroKubale1,GiaroKubale2,Hansen,Kampreprint,KamDiss,Kubale,PetDM,PetKarapet,PetKhachTan,PetKhachYepTan,Seva},
the problems of existence, construction and estimating the numerical
parameters of interval colorings of graphs were investigated.
Surveys on this topic can be found in some books
\cite{AsrDenHag,JensenToft,Kubale}.

Graph products \cite{HamImrichKlavzar} were first introduced by
Berge \cite{Berge}, Sabidussi \cite{Sabidussi}, Harary \cite{Harary}
and Vizing \cite{Vizing}. In particular, Sabidussi \cite{Sabidussi}
and Vizing \cite{Vizing} showed that every connected graph has a
unique decomposition into prime factors with respect to the
Cartesian product. In the same direction there are also many
interesting problems of decomposing of the different products of
graphs into Hamiltonian cycles. In particular, in \cite{BaranSas} it
was proved Bermond's conjecture that states: if two graphs are
decomposable into Hamiltonian cycles, then their composition is
decomposable, too. A lot of work was done on various topics related
to graph products, on the other hand there are still many questions
open. For example, it is still open Hedetniemi's conjecture
\cite{Hedetniemi}, Vizing's conjecture \cite{Vizingproblems} and the
conjecture of Harary, Kainen and Schwenk \cite{HarKainSchewk}.

There are many papers
\cite{HimmelWilliam,Jaradat,Kotzig,Mahamoodian,MoharPisanski,Mohar,PisanSHTMohar,Zhou}
devoted to proper edge-colorings of various products of graphs,
however very little is known on interval colorings of graph
products. Interval colorings of Cartesian products of graphs were
first investigated by Giaro and Kubale \cite{GiaroKubale1}. In
\cite{GiaroKubale2}, Giaro and Kubale proved that if $G,H\in
\mathfrak{N}$, then $G\square H\in \mathfrak{N}$. In 2004, they
formulated \cite{Kubale} a similar problem for the composition of
graphs as an open problem. In 2009, the first author \cite{PetDMGT}
showed that if $G,H\in \mathfrak{N}$ and $H$ is a regular graph,
then $G[H]\in \mathfrak{N}$. Later, Yepremyan \cite{PetKhachYepTan}
proved that if $G$ is a tree and $H$ is either a path or a star,
then $G[H]\in \mathfrak{N}$. Some other results on interval
colorings of various products of graphs were obtained in
\cite{Kubale,PetDMGT,PetKarapet,PetKhachTan,PetKhachYepTan}.

In this paper, we prove that if $G\in \mathfrak{N}$ and $H$ has an
interval coloring of a special type, then $G[H]\in \mathfrak{N}$.
Moreover, we show that all regular graphs, complete bipartite graphs
and trees have such a special interval coloring. In particular, this
implies that if $G\in \mathfrak{N}$ and $T$ is a tree, then $G[T]\in
\mathfrak{N}$.\\

\section{Notations, Definitions and Auxiliary Results}\

We use standard notations $C_{n}$ and $K_{n}$ for the simple cycle
and complete graph on $n$ vertices, respectively. We also use
standard notations $K_{m,n}$ and $K_{m,n,l}$ for the complete
bipartite and tripartite graph, respectively, one part of which has
$m$ vertices, the other part has $n$ vertices and the third part has
$l$ vertices.

For two positive integers $a$ and $b$ with $a\leq b$, we denote by
$\left[a,b\right]$ the interval of integers
$\left\{a,\ldots,b\right\}$.

Let $L=\left(l_{1},\ldots,l_{k}\right)$ be an ordered sequence of
nonnegative integers. The smallest and largest elements of $L$ are
denoted by $\underline L$ and $\overline L$, respectively. The
length (the number of elements) of $L$ is denoted by $\vert L\vert$.
By $L(i)$, we denote the $i$th element of $L$ ($1\leq i\leq k$). An
ordered sequence $L=\left(l_{1},\ldots,l_{k}\right)$ is called a
\emph{continuous sequence} if it contains all integers between
$\underline L$ and $\overline L$. If
$L=\left(l_{1},\ldots,l_{k}\right)$ is an ordered sequence and $p$
is nonnegative integer, then the sequence
$\left(l_{1}+p,\ldots,l_{k}+p\right)$ is denoted by $L\oplus p$.
Clearly, $(L\oplus p)(i)=L(i)+p$ for any $p\in \mathbb{Z}_{+}$.

Let $G$ and $H$ be two graphs. The composition (lexicographic
product) $G[H]$ of graphs $G$ and $H$ is defined as follows:
\begin{center}
$V(G[H])=V(G)\times V(H)$,
$E(G[H])=\left\{(u_{1},v_{1})(u_{2},v_{2})\colon\,u_{1}u_{2}\in
E(G)\vee (u_{1}=u_{2}\wedge v_{1}v_{2}\in E(H))\right\}$.
\end{center}

A \emph{partial edge-coloring} of $G$ is a coloring of some of the
edges of $G$ such that no two adjacent edges receive the same color.
If $\alpha $ is a proper edge-coloring of $G$ and $v\in V(G)$, then
$S\left(v,\alpha \right)$ (\emph{spectrum} of a vertex $v$) denotes
the set of colors appearing on edges incident to $v$. The smallest
and largest colors of $S\left(v,\alpha \right)$ are denoted by
$\underline S\left(v,\alpha \right)$ and $\overline S\left(v,\alpha
\right)$, respectively. A proper edge-coloring $\alpha$ of $G$ with
consecutive integers $c_{1},\ldots,c_{t}$ is called an
\emph{interval $t$-coloring} if all colors are used, and for any
$v\in V(G)$, the set $S\left(v,\alpha \right)$ is an interval of
integers. A graph $G$ is \emph{interval colorable} if it has an
interval $t$-coloring for some positive integer $t$. The set of all
interval colorable graphs is denoted by $\mathfrak{N}$. For a graph
$G\in \mathfrak{N}$, the smallest and the largest values of $t$ for
which it has an interval $t$-coloring
are denoted by $w(G)$ and $W(G)$, respectively.\\

In \cite{AsrKam,AsrKamJCTB}, Asratian and Kamalian obtained the
following result.

\begin{theorem}
\label{mytheorem1} If $G\in \mathfrak{N}$, then
$\chi^{\prime}(G)=\Delta(G)$. Moreover, if $G$ is a regular graph,
then $G\in \mathfrak{N}$ if and only if
$\chi^{\prime}(G)=\Delta(G)$.
\end{theorem}

In \cite{Kampreprint}, Kamalian proved the following result on
complete bipartite graphs.

\begin{theorem}
\label{mytheorem2} For any $m,n\in \mathbb{N}$, the complete
bipartite graph $K_{m,n}$ is interval colorable, and
\begin{description}
\item[(1)] $w\left(K_{m,n}\right)=m+n-\gcd(m,n)$,

\item[(2)] $W\left(K_{m,n}\right)=m+n-1$,

\item[(3)] if $w\left(K_{m,n}\right)\leq t\leq W\left(K_{m,n}\right)$, then $K_{m,n}$
has an interval $t$-coloring.
\end{description}
\end{theorem}

In \cite{Konig}, K\"onig proved the following result on bipartite
graphs.

\begin{theorem}
\label{mytheorem3} If $G$ is a bipartite graph, then
$\chi^{\prime}(G)=\Delta(G)$.
\end{theorem}

Let $\alpha$ be a proper edge-coloring of $G$ and
$V^{\prime}=\{v_{1},\ldots,v_{k}\}\subseteq V(G)$. Consider the sets
$S\left(v_{1},\alpha \right),\ldots,S\left(v_{k},\alpha \right)$.
For a coloring $\alpha$ of $G$ and $V^{\prime}\subseteq V(G)$,
define two ordered sequences $LSE(V^{\prime},\alpha)$ (\emph{Lower
Spectral Edge}) and $USE(V^{\prime},\alpha)$ (\emph{Upper Spectral
Edge}) as follows:
\begin{center}
$LSE(V^{\prime},\alpha)=\left(\underline S\left(v_{i_{1}},\alpha
\right),\underline S\left(v_{i_{2}},\alpha \right),\dots,\underline
S\left(v_{i_{k}},\alpha \right)\right)$,
\end{center}
where $\underline S\left(v_{i_{l}},\alpha \right)\leq \underline
S\left(v_{i_{l+1}},\alpha \right)$ for $1\leq l\leq k-1$, and
\begin{center}
$USE(V^{\prime},\alpha)=\left(\overline S\left(v_{j_{1}},\alpha
\right),\overline S\left(v_{j_{2}},\alpha \right),\dots,\overline
S\left(v_{j_{k}},\alpha \right)\right)$,
\end{center}
where $\overline S\left(v_{j_{l}},\alpha \right)\leq \overline
S\left(v_{j_{l+1}},\alpha \right)$ for $1\leq l\leq k-1$.

\begin{figure}[h]
\begin{center}
\includegraphics[width=13pc,height=10pc]{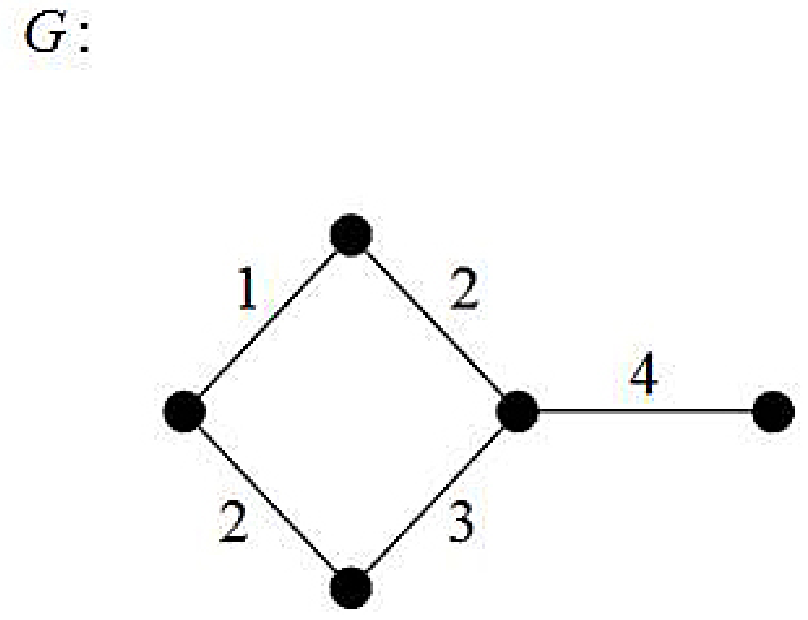}\\
\caption{The graph $G$ with its coloring $\alpha$ and with
$LSE(V(G),\alpha)=(1,1,2,2,4)$,
$USE(V(G),\alpha)=(2,2,3,4,4)$.}\label{Example1}
\end{center}
\end{figure}

For example, if we consider the graph $G$ with its coloring $\alpha$
shown in Fig. \ref{Example1}, then $LSE(V(G),\alpha)=(1,1,2,2,4)$
and $USE(V(G),\alpha)=(2,2,3,4,4)$. Moreover, the sequence
$(1,1,2,2,4)$ is not continuous, but the sequence $(2,2,3,4,4)$ is
continuous.

Recall that for ordered sequences $LSE(V^{\prime},\alpha)$ and
$USE(V^{\prime},\alpha)$, the number of elements in
$LSE(V^{\prime},\alpha)$ and $USE(V^{\prime},\alpha)$ is denoted by
$\vert LSE(V^{\prime},\alpha)\vert$ and $\vert
USE(V^{\prime},\alpha)\vert$, respectively. Clearly, $\vert
LSE(V(G),\alpha)\vert=\vert USE(V(G),\alpha)\vert=\vert
V(G)\vert$.\\

We also need the following lemma.

\begin{figure}[h]
\begin{center}
\includegraphics[width=20pc,height=20pc]{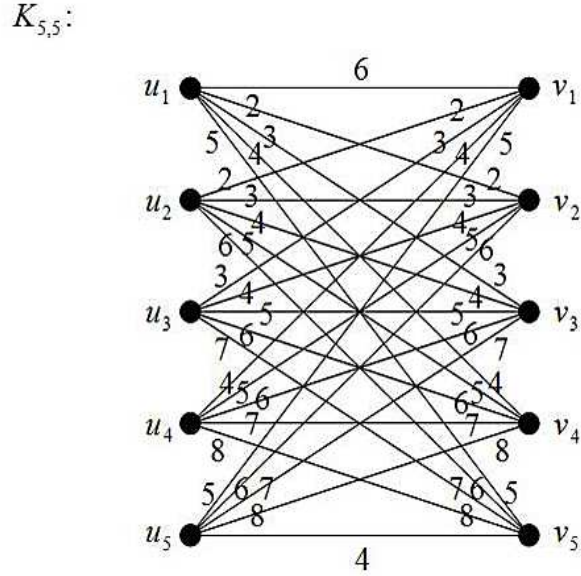}\\
\caption{The interval coloring $\gamma$ of $K_{5,5}$ with
$LSE(U,\gamma)=LSE(V,\gamma)=(2,2,3,4,4)$}\label{Example2}
\end{center}
\end{figure}

\begin{lemma}
\label{mylemma} If $K_{n,n}$ is a complete bipartite graph with
bipartition $(U,V)$, then for any continuous sequence $L$ with
length $n$, $K_{n,n}$ has an interval coloring $\alpha$ such that
\begin{center}
$LSE(U,\alpha)=LSE(V,\alpha)=L$.
\end{center}
\end{lemma}
\begin{proof} Let $K_{n,n}$ be a complete bipartite graph with
bipartition $(U,V)$, where $U=\{u_{1},\ldots,u_{n}\}$ and
$V=\{v_{1},\ldots,v_{n}\}$. Also, let
$L=\left(\underbrace{l_{1},\ldots,l_{1}}_{n_{1}},\underbrace{l_{2},\ldots,l_{2}}_{n_{2}},\ldots,\underbrace{l_{k},\ldots,l_{k}}_{n_{k}}\right)$
be a continuous sequence with length $n$
$\left(\sum_{i=1}^{k}n_{i}=n\right)$. Clearly, $l_{i+1}=l_{i}+1$ for
$1\leq l\leq k-1$.

First we define a partial edge-coloring $\alpha$ of $K_{n,n}$ as
follows:

\begin{description}
\item[1)] for $1\leq i\leq k-1$ and $p+q=1+\sum_{j=1}^{i}n_{j}$, let
$\alpha \left(u_{p}v_{q}\right)=l_{i}$;

\item[2)] for $1\leq i\leq k-1$ and $p+q=n+1+\sum_{j=1}^{i}n_{j}$, let
$\alpha \left(u_{p}v_{q}\right)=l_{i}+n$.
\end{description}

Define a subgraph $G$ of $K_{n,n}$ as follows:
\begin{center}
$V(G)=V(K_{n,n})$ and $E(G)=\left\{e\colon\,e\in E(K_{n,n})\wedge
\alpha(e)\in [l_{1},l_{k-1}]\cup [l_{1}+n,l_{k-1}+n]\right\}$.
\end{center}

By the definition of $\alpha$, $G$ is a spanning $(k-1)$-regular
bipartite subgraph of $K_{n,n}$. Next we define a subgraph
$G^{\prime}$ of $K_{n,n}$ as follows:
\begin{center}
$V(G)=V(K_{n,n})$ and
$E\left(G^{\prime}\right)=E\left(K_{n,n}\right)\setminus E(G)$.
\end{center}

Clearly, $G^{\prime}$ is a spanning $(n-k+1)$-regular bipartite
subgraph of $K_{n,n}$. By Theorem \ref{mytheorem3},
$\chi^{\prime}\left(G^{\prime}\right)=\Delta
\left(G^{\prime}\right)=n-k+1$. Let $\beta$ be a proper
edge-coloring of $G^{\prime}$ with colors
$l_{k},l_{k}+1,\ldots,l_{k}+n-k$. By the definition of $\beta$, for
each vertex $v\in V(K_{n,n})$, $S(v,\beta)=[l_{k},l_{k}+n-k]$.

Now we are able to define an edge-coloring $\gamma$ of $K_{n,n}$.

For every $e\in E(K_{n,n})$, let

\begin{center}
$\gamma(e)=\left\{
\begin{tabular}{ll}
$\alpha(e)$, & if $e\in E(G)$,\\
$\beta(e)$, & if $e\in E\left(G^{\prime}\right)$.\\
\end{tabular}%
\right.$
\end{center}

Let us prove that $\gamma$ is an interval $(l_{k}+n-1)$-coloring of
$K_{n,n}$ such that $S(u_{i},\gamma)=S(v_{i},\gamma)$ and
$\underline S(u_{i},\gamma)=\underline S(v_{i},\gamma)=l_{i}$ for
$1\leq i\leq n$.

By the definition of $\gamma$, for $1\leq i\leq n$, we have
\begin{center}
$S(u_{i},\gamma)=S(v_{i},\gamma)=[l_{1},l_{1}+n-1]$ if $i\in
\left[1,n_{1}\right]$,

$S(u_{i},\gamma)=S(v_{i},\gamma)=[l_{2},l_{2}+n-1]$ if $i\in
\left[n_{1}+1,n_{1}+n_{2}\right]$,

$\ldots$

$S(u_{i},\gamma)=S(v_{i},\gamma)=[l_{k},l_{k}+n-1]$ if $i\in
\left[\sum_{j=1}^{k-1}n_{j}+1,\sum_{j=1}^{k}n_{j}\right]$.
\end{center}

This implies that $\gamma$ is an interval $(l_{k}+n-1)$-coloring of
$K_{n,n}$ and $LSE(U,\gamma)=LSE(V,\gamma)=L$. ~$\square$
\end{proof}

Fig. \ref{Example2} shows the interval coloring $\gamma$ of
$K_{5,5}$ described in the proof of Lemma \ref{mylemma}.\\

\bigskip

\section{The Main Result}\

Here, we prove our main result which states that if $G\in
\mathfrak{N}$ and $H$ has an interval coloring of a special type,
then $G[H]\in \mathfrak{N}$.

\begin{theorem}
\label{mytheorem4} If $G\in \mathfrak{N}$ and $H$ has an interval
coloring $\alpha_{H}$ such that $USE(V(H),\alpha_{H})$ is
continuous, then $G[H]\in \mathfrak{N}$. Moreover, if $\vert
V(H)\vert = n$ and $L=USE(V(H),\alpha_{H})$, then
\begin{center}
$w\left(G[H]\right)\leq w(G)\cdot n + \overline L$ and
$W\left(G[H]\right)\geq W(G)\cdot n + \overline L$
\end{center}
\end{theorem}
\begin{proof}
Let $V(G) = \left\{u_{1},\ldots,u_{m}\right\}$, $V(H) =
\left\{w_{1},\ldots,w_{n}\right\}$ and

\begin{center}
$V(G[H]) = \left\{v_{j}^{(i)}\colon\,1\leq i\leq m, 1\leq j \leq
n\right\}$ and $E(G[H])
=\left\{v_{p}^{(i)}v_{q}^{(j)}\colon\,u_{i}u_{j}\in E(G), 1\leq
p\leq n, 1\leq q\leq n\right\}\cup \bigcup_{i=1}^{m}{E^{i}}$,
\end{center}
where $E^{i} = \left\{v_{p}^{(i)}v_{q}^{(i)}\colon\,w_{p}w_{q}\in
E(H)\right\}$.

Let $\alpha_{G}$ be an interval $t$-coloring of $G$ and $L$ be a
continuous sequence with length $n$ such that
$L=USE(V(H),\alpha_{H})$. Without loss of generality we may assume
that vertices of $H$ are numbered so that $\overline
S\left(w_{i},\alpha_{H}\right)=L(i)$ for $1\leq i\leq n$. Let us
consider the graph $K_{2}[H]$. Clearly, $K_{2}[H]$ is isomorphic to
$K_{n,n}$. Let
$V\left(K_{2}[H]\right)=\left\{x_{1},\ldots,x_{n},y_{1},\ldots,y_{n}\right\}$
and $E\left(K_{2}[H]\right)=\left\{x_{i}y_{j}\colon\, 1\leq i\leq n,
1\leq j\leq n\right\}$. Since $L$ is a continuous sequence, $L
\oplus 1$ is a continuous sequence, too. By Lemma \ref{mylemma},
$K_{2}[H]$ has an interval coloring $\beta$ such that $\underline
S\left(x_{i},\beta\right)=\underline
S\left(y_{i},\beta\right)=L(i)+1$ for $1\leq i\leq n$.\\

Now we are able to define an edge-coloring $\alpha_{G[H]}$ of
$G[H]$.

\begin{description}
\item[1)] For $1\leq i\leq m$ and $v_{p}^{(i)}v_{q}^{(i)}\in E^{i}$ ($p,q=1,\ldots n$), let
\begin{center}
$\alpha_{G[H]}\left(v_{p}^{(i)}v_{q}^{(i)}\right)=\left(\underline
S\left(u_{i},\alpha_{G}\right)-1\right)n +
\alpha_{H}\left(w_{p}w_{q}\right)$.
\end{center}

\item[2)] For $1\leq i<j\leq m$ and $v_{p}^{(i)}v_{q}^{(j)}\in E(G[H])$ ($p,q=1,\ldots n$), let
\begin{center}
$\alpha_{G[H]}\left(v_{p}^{(i)}v_{q}^{(j)}\right)=\left(
\alpha_{G}\left(u_{i}u_{j}\right)-1\right)n +
\beta\left(x_{p}y_{q}\right)$.
\end{center}
\end{description}

It is not difficult to see that $\alpha_{G[H]}$ is a proper
edge-coloring of $G[H]$. Let us prove that $\alpha_{G[H]}$ is an
interval $(t\cdot n + \overline L)$-coloring of $G[H]$. For the
proof, it suffices to show that for $1\leq i\leq m$ and $1\leq j\leq
n$,
\begin{center}
$\overline S\left(v_{j}^{(i)},\alpha_{G[H]}\right)-\underline
S\left(v_{j}^{(i)},\alpha_{G[H]}\right)=d_{G[H]}\left(v_{j}^{(i)}\right)-1$.
\end{center}

By the definition of $\alpha_{G[H]}$, for $1\leq i\leq m$ and $1\leq
j\leq n$, we have
\begin{center}
$\overline S\left(v_{j}^{(i)},\alpha_{G[H]}\right)=\left(\overline
S\left(u_{i},\alpha_{G}\right)-1\right)n + L(j)+1+n-1=\overline
S\left(u_{i},\alpha_{G}\right)\cdot n+L(j)$.
\end{center}

By the definition of $\alpha_{G[H]}$ and taking into account that
$L(j)-\underline
S\left(w_{j},\alpha_{H}\right)=d_{H}\left(w_{j}\right)-1$ ($1\leq
j\leq n$), for $1\leq i\leq m$ and $1\leq j\leq n$, we have
\begin{center}
$\underline S\left(v_{j}^{(i)},\alpha_{G[H]}\right)=\left(\underline
S\left(u_{i},\alpha_{G}\right)-1\right)n +
L(j)-d_{H}\left(w_{j}\right)+1$.
\end{center}

Now, taking into account that $\overline
S\left(u_{i},\alpha_{G}\right)-\underline
S\left(u_{i},\alpha_{G}\right)=d_{G}(u_{i})-1$ ($1\leq i\leq m$),
for $1\leq i\leq m$ and $1\leq j\leq n$, we obtain
\begin{center}
$\overline S\left(v_{j}^{(i)},\alpha_{G[H]}\right)-\underline
S\left(v_{j}^{(i)},\alpha_{G[H]}\right)=\left(\overline
S\left(u_{i},\alpha_{G}\right)-\underline
S\left(u_{i},\alpha_{G}\right)+1\right)n+d_{H}\left(w_{j}\right)-1=d_{G}\left(u_{i}\right)\cdot
n+d_{H}\left(w_{j}\right)-1=d_{G[H]}\left(v_{j}^{(i)}\right)-1$.
\end{center}

This shows that $\alpha_{G[H]}$ is an interval $(t\cdot n +
\overline L)$-coloring of $G[H]$. Thus, $w\left(G[H]\right)\leq
w(G)\cdot n + \overline L$ and $W\left(G[H]\right)\geq W(G)\cdot n +
\overline L$. ~$\square$
\end{proof}

\begin{corollary}
\label{mycorollary1} If $G,H\in \mathfrak{N}$ and $H$ is an
$r$-regular graph, then $G[H]\in \mathfrak{N}$. Moreover, if $\vert
V(H)\vert = n$, then
\begin{center}
$w(G[H])\leq w(G)\cdot n + r$ and $W(G[H])\geq W(G)\cdot n + r$.
\end{center}
\end{corollary}
\begin{proof} Since $H\in \mathfrak{N}$ and $H$ is an
$r$-regular graph, by Theorem \ref{mytheorem1},
$\chi^{\prime}(H)=\Delta(H)=r$. This implies that $H$ has a proper
edge-coloring $\alpha_{H}$ with colors $1,\ldots,r$. Hence, for
every $v\in V(H)$, $S\left(v,\alpha_{H}\right)=[1,r]$. Clearly,
$\alpha_{H}$ is an interval $r$-coloring and
$USE(V(H),\alpha_{H})=(r,\ldots,r)$ is continuous, so, by Theorem
\ref{mytheorem4}, $G[H]\in \mathfrak{N}$. Moreover, if $\vert
V(H)\vert = n$, then $w(G[H])\leq w(G)\cdot n + r$ and $W(G[H])\geq
W(G)\cdot n + r$. ~$\square$
\end{proof}

\begin{corollary}
\label{mycorollary2} If $G\in \mathfrak{N}$, then $G[\overline
K_{n}]\in \mathfrak{N}$ for any $n\in \mathbb{N}$. Moreover,
$w(G[\overline K_{n}])\leq w(G)\cdot n$ and $W(G[\overline
K_{n}])\geq W(G)\cdot n$.
\end{corollary}
\begin{proof} We may assume that $\overline
K_{n}$ has an interval coloring $\alpha$ such that $USE(V(\overline
K_{n}),\alpha)=(0,\ldots,0)$. Since $USE(V(\overline
K_{n}),\alpha)=(0,\ldots,0)$ is continuous, by Theorem
\ref{mytheorem4}, $G[\overline K_{n}]\in \mathfrak{N}$. Moreover,
$w(G[\overline K_{n}])\leq w(G)\cdot n$ and $W(G[\overline
K_{n}])\geq W(G)\cdot n$. ~$\square$
\end{proof}

Fig. \ref{Example3} shows the interval $14$-coloring
$\alpha_{P_{3}[H]}$ of $P_{3}[H]$ described in the proof of Theorem \ref{mytheorem4}.\\

\begin{figure}[h]
\begin{center}
\includegraphics[width=40pc,height=20pc]{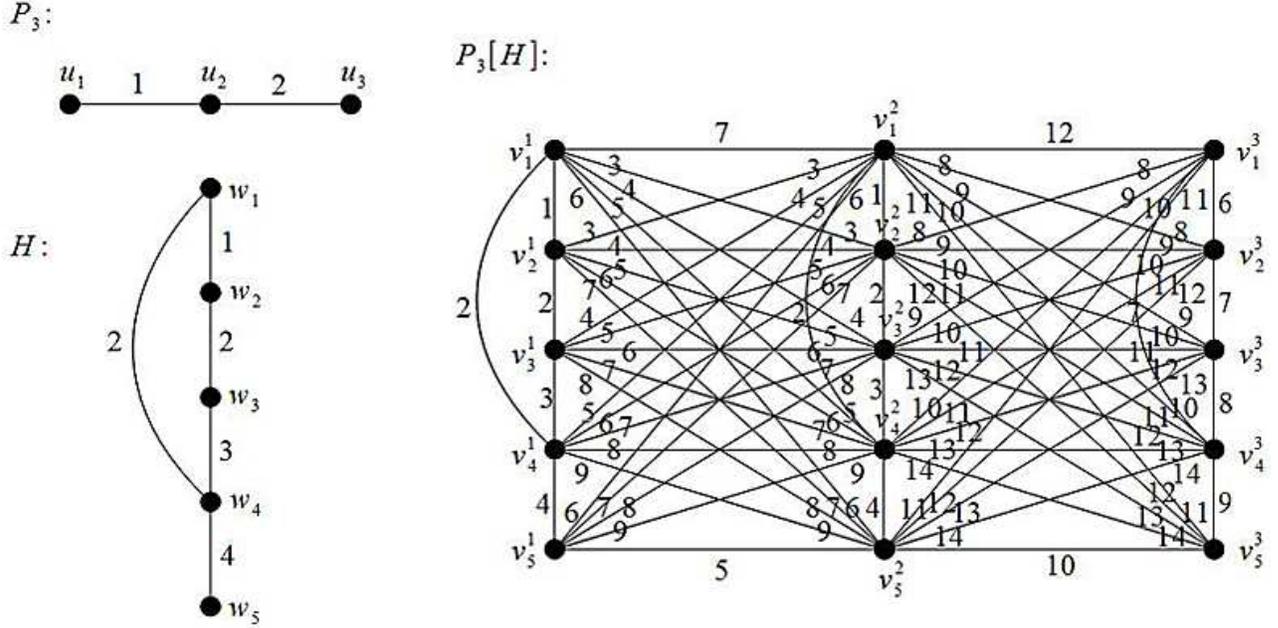}\\
\caption{The interval $14$-coloring $\alpha_{P_{3}[H]}$ of
$P_{3}[H]$.}\label{Example3}
\end{center}
\end{figure}

\bigskip

\section{Applications of the Main Result}\

This section is devoted to applications of the main result from the
previous section for some classes of graphs. We first consider
complete bipartite graphs.

\begin{theorem}
\label{mytheorem5} If $G\in \mathfrak{N}$, then $G[K_{m,n}]\in
\mathfrak{N}$ for any $m,n\in \mathbb{N}$. Moreover, for any $m,n\in
\mathbb{N}$, we have
\begin{center}
$w\left(G[K_{m,n}]\right)\leq (w(G)+1)(m+n)-1$ and
$W\left(G[K_{m,n}]\right)\geq (W(G)+1)(m+n)-1$.
\end{center}
\end{theorem}
\begin{proof} Let $(U,V)$ be a bipartition of $K_{m,n}$, where
$U=\{u_{1},\ldots,u_{m}\}$ and $V=\{v_{1},\ldots,v_{n}\}$. Define an
edge-coloring $\alpha$ of $K_{m,n}$ as follows: for each edge
$u_{i}v_{j}\in E(K_{m,n})$, let $\alpha(u_{i}v_{j})=i+j-1$, where
$1\leq i\leq m$, $1\leq j\leq n$. Clearly, $\alpha$ is an interval
$(m+n-1)$-coloring of $K_{m,n}$. Moreover,
$S(u_{i},\alpha)=[i,i+n-1]$ for $1\leq i\leq m$ and
$S(v_{j},\alpha)=[j,j+m-1]$ for $1\leq j\leq n$. This implies that
$USE(U,\alpha)=(n,n+1,\ldots,m+n-1)$ and
$USE(V,\alpha)=(m,m+1,\ldots,m+n-1)$. Since
$USE(V\left(K_{m,n}\right),\alpha)$ is the union of $USE(U,\alpha)$
and $USE(V,\alpha)$, we obtain $USE(V\left(K_{m,n}\right),\alpha)$
is a continuous sequence. By Theorem \ref{mytheorem4},
$G[K_{m,n}]\in \mathfrak{N}$. Moreover, $w(G[K_{m,n}])\leq w(G)\cdot
(m+n)+m+n-1$ and $W(G[K_{m,n}])\geq W(G)\cdot (m+n)+m+n-1$.
~$\square$
\end{proof}

Next, we consider complete graphs of even order. Here we need one
result on interval colorings of complete graphs of even order. In
\cite{PetDM}, it was proved the following result.

\begin{theorem}
\label{mytheorem6} For any $n\in \mathbb{N}$, $K_{2n}$ has an
interval $(3n-2)$-coloring $\alpha$ such that for each $i\in [1,n]$,
there are vertices $v_{i}^{\prime },v_{i}^{\prime \prime}\in
V\left(K_{2n}\right)$ $\left(v_{i}^{\prime }\neq v_{i}^{\prime\prime
}\right)$ with $\underline S\left(v_{i}^{\prime },\alpha \right)
=\underline S\left(v_{i}^{\prime\prime },\alpha \right)=i$.
\end{theorem}

Now we are able to prove our result on complete graphs of even
order.

\begin{theorem}
\label{mytheorem7} If $G\in \mathfrak{N}$, then $G[K_{2n}]\in
\mathfrak{N}$ for any $n\in \mathbb{N}$. Moreover, for any $n\in
\mathbb{N}$, we have
\begin{center}
$w\left(G[K_{2n}]\right)\leq (2\cdot w(G)+2)n-1$ and
$W\left(G[K_{2n}]\right)\geq (2\cdot W(G)+3)n-2$.
\end{center}
\end{theorem}
\begin{proof} By Corollary \ref{mycorollary1}, if $G\in \mathfrak{N}$, then $G[K_{2n}]\in
\mathfrak{N}$ and $w\left(G[K_{2n}]\right)\leq w(G)\cdot 2n+2n-1$
for any $n\in \mathbb{N}$.

Now we show that $W\left(G[K_{2n}]\right)\geq (2\cdot W(G)+3)n-2$.
By Theorem \ref{mytheorem6}, $K_{2n}$ has an interval
$(3n-2)$-coloring $\alpha$ such that for each $i\in [1,n]$, there
are vertices $v_{i}^{\prime },v_{i}^{\prime \prime}\in
V\left(K_{2n}\right)$ $\left(v_{i}^{\prime }\neq v_{i}^{\prime\prime
}\right)$ with $S\left(v_{i}^{\prime },\alpha \right) =
S\left(v_{i}^{\prime\prime },\alpha \right)=[i,i+2n-2]$. This
implies that
$USE(V(K_{2n}),\alpha)=(2n-1,2n-1,2n,2n,\ldots,3n-2,3n-2)$, which is
a continuous sequence. By Theorem \ref{mytheorem4}, $G[K_{2n}]\in
\mathfrak{N}$ and $W(G[K_{2n}])\geq W(G)\cdot 2n+3n-2$. ~$\square$
\end{proof}

A similar result also can be obtained for even cycles.

\begin{theorem}
\label{mytheorem8} If $G\in \mathfrak{N}$, then $G[C_{2n}]\in
\mathfrak{N}$ for any integer $n\geq 2$. Moreover, for any integer
$n\geq 2$, we have
\begin{center}
$w\left(G[C_{2n}]\right)\leq 2(w(G)\cdot n+1)$ and
$W\left(G[C_{2n}]\right)\geq (2\cdot W(G)+1)n+1$.
\end{center}
\end{theorem}
\begin{proof} By Corollary \ref{mycorollary1}, if $G\in \mathfrak{N}$, then $G[C_{2n}]\in
\mathfrak{N}$ and $w\left(G[C_{2n}]\right)\leq w(G)\cdot 2n+2$ for
any integer $n\geq 2$.

Now we show that $W\left(G[C_{2n}]\right)\geq (2\cdot W(G)+1)n+1$.
Let $V(C_{2n})=\{v_{1},\ldots,v_{2n}\}$ and $E(C_{2n}) =\{
v_{i}v_{i+1}\colon\,1\leq i\leq 2n-1\}\cup \{v_{1}v_{2n}\}$. Define
an edge-coloring $\alpha$ of $C_{2n}$ as follows: for $1\leq i\leq
n$, let $\alpha(v_{i}v_{i+1})=\alpha(v_{2n+1-i}v_{2n-i})=i+1$ and
$\alpha(v_{1}v_{2n})=1$. Clearly, $\alpha$ is an interval
$(n+1)$-coloring of $C_{2n}$ such that for each $i\in [1,n]$,
$S\left(v_{i},\alpha \right) = S\left(v_{2n+1-i},\alpha
\right)=[i,i+1]$. This implies that
$USE(V(C_{2n}),\alpha)=(2,2,3,3,\ldots,n+1,n+1)$, which is a
continuous sequence. By Theorem \ref{mytheorem4}, $G[C_{2n}]\in
\mathfrak{N}$ and $W(G[C_{2n}])\geq W(G)\cdot 2n+n+1$. ~$\square$
\end{proof}

Finally, we show that every tree $T$ has an interval coloring
$\alpha$ such that $USE(V(T),\alpha)$ is continuous.

\begin{theorem}
\label{mytheorem8} If $T$ is a tree, then it has an interval
coloring $\alpha$ such that $USE(V(T),\alpha)$ is continuous.
\end{theorem}
\begin{proof}
Let $T$ be a tree with $\vert V(T)\vert =n$ ($n\geq 2$). We prove
the theorem by induction on $\vert E(T)\vert$. We will construct
tree $T$ starting from some $v_{1}v_{2}$ edge and adding a new leaf
on each step. For $1\leq i\leq n-1$, we denote by $T_{i}$ the tree
obtained on step $i$ and by $\alpha_{i}$ its edge-coloring. For a
tree $T_{i}$ and its edge-coloring $\alpha_{i}$ ($1\leq i\leq n-1$),
define numbers $a_{i}$ and $b_{i}$ as follows:
\begin{center}
$a_{i}=\min_{e\in E(T_{i})} \alpha_{i}(e)$ and $b_{i}=\max_{e\in
E(T_{i})} \alpha_{i}(e)$.
\end{center}

We show that in each step $T_{i}$ and $\alpha_{i}$ satisfy the
following two conditions:

\begin{description}
\item[(1)] for each $v\in V(T_{i})$, $S\left(v,\alpha_{i}\right)$ is an interval of integers;

\item[(2)] each color of the interval $\left[a_{i},b_{i}\right]$ appears in
$USE\left(V(T_{i}),\alpha_{i}\right)$.
\end{description}

Let $V(T_{1})=\{v_{1},v_{2}\}$ and $E(T_{1})=\{v_{1}v_{2}\}$. Define
an edge-coloring $\alpha_{1}$ of $T_{1}$ as follows:
$\alpha_{1}(v_{1}v_{2})=\vert E(T)\vert$. Since
$S\left(v_{1},\alpha_{1}\right)=S\left(v_{2},\alpha_{1}\right)=\{\vert
E(T)\vert\}$, we have $a_{1}=b_{1}=\vert E(T)\vert$ and
$USE\left(V(T_{1}),\alpha_{1}\right)=(\vert E(T)\vert,\vert
E(T)\vert)$. This implies that $(1)$ and $(2)$ hold for $T_{1}$.
Suppose that $n\geq 3$, $(1)$ and $(2)$ are satisfied for a tree
$T_{m-1}$ and its edge-coloring $\alpha_{m-1}$, and prove that $(1)$
and $(2)$ are also satisfied for a tree $T_{m}$ and its
edge-coloring $\alpha_{m}$ ($2\leq m\leq n-1$). Let $u$ be the
pendant vertex that should be added to $T_{m-1}$ to get $T_{m}$.
Also, let $uw\in E(T_{m})$, where $w\in V(T_{m-1})$.\\

Define an edge-coloring $\alpha_{m}$ of $T_{m}$ as follows: for
every $e\in E(T_{m})$, let

\begin{center}
$\alpha_{m}(e)=\left\{
\begin{tabular}{ll}
$\alpha_{m-1}(e)$, & if $e\in E(T_{m-1})$,\\
$\underline S\left(w,\alpha_{m-1}\right)-1$, & if $e=uw$.\\
\end{tabular}%
\right.$
\end{center}

By the definition of $\alpha_{m}$, we have:

\begin{description}
\item[1)] for each $v\in V(T_{m})$, $S\left(v,\alpha_{m}\right)$ is an interval of integers;

\item[2)] for $v\in V(T_{m-1})$, $\overline
S\left(v,\alpha_{m}\right)=\overline S\left(v,\alpha_{m-1}\right)$
and $USE\left(V(T_{m}),\alpha_{m}\right)$ is the union of
$USE\left(V(T_{m-1}),\alpha_{m-1}\right)$ and
$\left(\alpha_{m}(uw)\right)$;

\item[3)] $a_{m}=\min \{a_{m-1},\alpha_{m}(uw)\}, b_{m}=b_{m-1}$ and
$\alpha_{m}(uw)=\underline S\left(w,\alpha_{m-1}\right)-1\geq
a_{m-1}-1$.
\end{description}

By $1),2)$ and $3)$, and taking into account that each color of the
interval $\left[a_{m-1},b_{m-1}\right]$ appears in
$USE\left(V(T_{m-1}),\alpha_{m-1}\right)$, we obtain that each color
of the interval $\left[a_{m},b_{m}\right]$ appears in
$USE\left(V(T_{m}),\alpha_{m}\right)$. This implies that $(1)$ and
$(2)$ also hold for $T_{m}$. So, taking $m=n-1$, we get that
$T=T_{n-1}$. Finally, define an edge-coloring $\alpha$ of $T$ as
follows: for every $e\in E(T)$, let
$\alpha(e)=\alpha_{n-1}(e)-a_{n-1}+1$. It is not difficult to see
that $\alpha$ is an interval $\left(\vert E(T)\vert -
a_{n-1}+1\right)$-coloring of $T$ such that $USE(V(T),\alpha)$ is
continuous. ~$\square$
\end{proof}

\begin{corollary}
\label{mycorollary3} If $G\in \mathfrak{N}$ and $T$ is a tree, then
$G[T]\in \mathfrak{N}$.
\end{corollary}

\bigskip

\section{Concluding Remarks}\

In the previous sections it was proved that if $G\in \mathfrak{N}$
and $H$ has an interval coloring $\alpha_{H}$ such that
$USE(V(H),\alpha_{H})$ is continuous, then $G[H]\in \mathfrak{N}$.
Unfortunately, not all interval colorable graphs have such a special
interval coloring. For example, if we consider the complete
tripartite graph $K_{1,1,2n}$ ($n\geq 2$), then it is not difficult
to see that for every interval coloring $\alpha$ of $K_{1,1,2n}$
($n\geq 2$), $USE(V(K_{1,1,2n}),\alpha)$ is not continuous. This
implies that the problem on interval colorability of the composition
of interval colorable graphs still remains open.

\bigskip

\begin{acknowledgement}
We would like to thank Hrant Khachatrian for his constructive
suggestions on improvements of the paper.
\end{acknowledgement}


\bigskip

\begin{thebibliography}{99}

\bibitem{AsrKam} A.S. Asratian, R.R. Kamalian, Interval colorings of edges of a
multigraph, Appl. Math. 5 (1987) 25-34 (in Russian).

\bibitem{AsrKamJCTB} A.S. Asratian, R.R. Kamalian, Investigation on interval
edge-colorings of graphs, J. Combin. Theory Ser. B 62 (1994) 34-43.

\bibitem{AsrDenHag} A.S. Asratian, T.M.J. Denley, R. Haggkvist, Bipartite Graphs and
their Applications, Cambridge University Press, Cambridge, 1998.

\bibitem{BaranSas} Z. Baranyai, Gy. R. Szasz, Hamiltonian decomposition of
lexicographic product, J. Combin. Theory Ser. B 31 (1981) 253-261.

\bibitem{Berge} C. Berge, Theorie des Graphes et ses Applications, Dunod, Paris,
1958 (in French).

\bibitem{GiaroKubale1} K. Giaro, M. Kubale, Consecutive edge-colorings of
complete and incomplete Cartesian products of graphs, Cong, Num. 128
(1997) 143-149.

\bibitem{GiaroKubale2} K. Giaro, M. Kubale, Compact scheduling of zero-one time operations
in multi-stage systems, Discrete Appl. Math. 145 (2004) 95-103.

\bibitem{HamImrichKlavzar} R. Hammack, W. Imrich, S. Klavzar, Handbook of Product Graphs,
Second Edition, CRC Press, 2011.

\bibitem{Hansen} H.M. Hansen, Scheduling with minimum waiting periods, MSc
Thesis, Odense University, Odense, Denmark, 1992 (in Danish).

\bibitem{Harary} F. Harary, Graph Theory, Addison-Wesley, Reading MA, 1969.

\bibitem{HarKainSchewk} F. Harary, P.C. Kainen, A.J. Schwenk,
Toroidal graphs with arbitrarily high crossing numbers, Nanta Math.
6 (1973) 58-67.

\bibitem{Hedetniemi} S. Hedetniemi, Homomorphisms of graphs and automata, Doctoral
Dissertation, University of Michigan, 1966.

\bibitem{HimmelWilliam} P.E. Himmelwright, J.E. Williamson, On 1-factorability and
edge-colorability of cartesian products of graphs, Elem. Der Math.
29 (1974) 66-67.

\bibitem{Jaradat} M.M.M. Jaradat, On the Anderseon-Lipman conjecture and some related
problems, Discrete Math. 297 (2005) 167-173.

\bibitem{JensenToft} T.R. Jensen, B. Toft, Graph Coloring Problems, Wiley Interscience
Series in Discrete Mathematics and Optimization, 1995.

\bibitem{Kampreprint} R.R. Kamalian, Interval colorings of complete bipartite graphs
and trees, preprint, Comp. Cen. of Acad. Sci. of Armenian SSR,
Yerevan, 1989 (in Russian).

\bibitem{KamDiss} R.R. Kamalian, Interval edge colorings of graphs, Doctoral
Thesis, Novosibirsk, 1990.

\bibitem{Konig} D. K\"onig, \"Uber Graphen und ihre Anwendung auf
Determinantentheorie und Mengenlehre. Math. Ann. 77 (1916) 453-465.

\bibitem{Kotzig} A. Kotzig, 1-Factorizations of cartesian products of regular
graphs, J. Graph Theory 3 (1979) 23-34.

\bibitem{Kubale} M. Kubale, Graph colorings, American Mathematical Society, 2004.

\bibitem{Mahamoodian} E.S. Mahamoodian, On edge-colorability of cartesian products of
graphs, Canad. Math. Bull. 24 (1981) 107-108.

\bibitem{MoharPisanski} B. Mohar, T. Pisanski, Edge-coloring of a family of regular graphs,
Publ. Inst. Math. (Beograd) 33 (47) (1983) 157-162.

\bibitem{Mohar} B. Mohar, On edge-colorability of products of graphs, Publ. Inst.
Math. (Beograd) 36 (50) (1984) 13-16.

\bibitem{PetDM} P.A. Petrosyan, Interval edge-colorings of complete graphs and
$n$-dimensional cubes, Discrete Math. 310 (2010) 1580-1587.

\bibitem{PetDMGT} P.A. Petrosyan, Interval edge colorings of some products of graphs,
Discuss. Math. Graph Theory 31(2) (2011) 357-373.

\bibitem{PetKarapet} P.A. Petrosyan, G.H. Karapetyan, Lower bounds for the greatest
possible number of colors in interval edge colorings of bipartite
cylinders and bipartite tori, in: Proceedings of the CSIT Conference
(2007) 86-88.

\bibitem{PetKhachTan} P.A. Petrosyan, H.H. Khachatrian, H.G. Tananyan, Interval
edge-colorings of Cartesian products of graphs I, Discuss. Math.
Graph Theory 33(3) (2013) 613-632.

\bibitem{PetKhachYepTan} P.A. Petrosyan, H.H. Khachatrian, L.E. Yepremyan, H.G. Tananyan,
Interval edge-colorings of graph products, in: Proceedings of the
CSIT Conference (2011) 89-92.

\bibitem{PisanSHTMohar} T. Pisanski, J. Shawe-Taylor, B. Mohar, 1-Factorization of the
composition of regular graphs, Publ. Inst. Math. (Beograd) 33 (47)
(1983) 193-196.

\bibitem{Sabidussi} G. Sabidussi, Graph multiplication, Math. Z. 72 (1960) 446-457.

\bibitem{Seva} S.V. Sevast'janov, Interval colorability of the edges of a
bipartite graph, Metody Diskret. Analiza 50 (1990) 61-72 (in
Russian).

\bibitem{Vizing} V.G. Vizing, The Cartesian product of graphs, Vych. Sis. 9 (1963)
30-43 (in Russian).

\bibitem{Vizingproblems} V.G. Vizing, Some unsolved problems in graph theory, Uspehi Mat.
Nauk 23 (1968) 117-134 (in Russian).

\bibitem{Zhou} M.K. Zhou, Decomposition of some product graphs into 1-factors and
Hamiltonian cycles, Ars Combin. 28 (1989) 258-268. 23.

\bibitem{West} D.B. West, Introduction to Graph Theory, Prentice-Hall, New
Jersey, 2001.

\end{thebibliography}
\end{document}